\documentclass[11pt]{article}

\usepackage[english]{babel}
\usepackage{amsfonts,amsmath,amssymb,amsthm}
\usepackage{nicefrac}
\usepackage{url}
\usepackage{ifthen}
\usepackage{verbatim}
\usepackage{mathrsfs}
\usepackage{textcomp}
\usepackage{ifpdf}
\usepackage{todonotes}
\usepackage{aeguill}
\usepackage[top = 3.4cm, bottom = 3.4cm, left = 3.4cm, right = 3.4cm]{geometry}

\newcommand{\setN}{\ensuremath{\mathbb{N}}}

\newcommand{\setR}{\ensuremath{\mathbb{R}}}
\newcommand{\setZ}{\ensuremath{\mathbb{Z}}}

\newcommand{\Dcal}{\ensuremath{\mathcal{D}}}

\newcommand{\Pcal}{\ensuremath{\mathcal{P}}}

\newcommand{\set}[2][]{#1\{ {#2} #1\}}
\newcommand{\card}[2][]{#1| #2 #1|}
\newcommand{\paren}[2][]{#1( #2 #1)}
\newcommand{\ceil}[2][]{#1\lceil #2 #1\rceil}

\newcommand{\st}{\colon}

\DeclareMathOperator{\funcaoprob}{\mathbb{P}}
\newcommand{\prob}[2][]{\mathbb{P}#1(#2#1)}

\DeclareMathOperator{\probmultisub}{multi}

\newcommand{\probmultiop}{\funcaoprob_{\probmultisub}}
\newcommand{\probsimpleop}{\funcaoprob}
\newcommand{\probmulti}[2][]{\probmultiop#1(#2#1)}
\newcommand{\probsimple}[2][]{\probsimpleop#1(#2#1)}

\newcommand{\probcond}[3][]{\mathbb{P}#1(#2\,#1|\,#3#1)}

\newcommand{\esp}[2][]{\mathbb{E}\,#1(#2#1)}

\newcommand{\var}[2][]{\mathop{\rm Var}#1(#2#1)}

\newcommand{\tpoisson}[3][]{\mathop{\rm Po}#1(#2, #3 #1)}

\newcommand{\multinomial}[4][]{\mathop{\rm Multi}_{\geq #2} #1(#3, #4 #1)}
\newcommand{\multinomialknm}[1][]{\multinomial[#1]{k}{n}{m}}

\newtheorem{thm}{Theorem}[section]
\newtheorem{prop}[thm]{Proposition}
\newtheorem{lem}[thm]{Lemma}
\newtheorem{cor}[thm]{Corollary}

  {%
   \end{list}%
   \smallskip
 }
 
\newcommand{\Bollobas}{Bollob{\'a}s}

\newcommand{\eps}{\varepsilon}
\newcommand{\aas}{a.a.s.}

\newcommand{\Dknm}{\ensuremath{\Dcal_k(n,m)}}

\newcommand{\tDknm}{\ensuremath{\tilde\Dcal_k(n,m)}}
\newcommand{\hDknm}{\ensuremath{\hat\Dcal_k(n,m)}}

\newcommand{\cDknm}{\ensuremath{\check\Dcal_k(n,m)}}

\newcommand{\ds}{\mathbf{d}}
\newcommand{\hs}{\mathbf{h}}
\newcommand{\js}{\mathbf{j}}
\newcommand{\gs}{\mathbf{g}}

\newcommand{\Ys}{\mathbf{Y}}

\DeclareMathOperator{\multisub}{multi}
\newcommand{\multi}{\ensuremath{\mathop{G_{\multisub}}}}
\newcommand{\multipar}[4][]{\multi #1(#2, #3, #4 #1)}
\newcommand{\multiknm}[1][]{\multipar[#1]{k}{n}{m}}

\newcommand{\simple}{\ensuremath{\mathop{G}}}
\newcommand{\simplepar}[4][]{\ensuremath{\simple #1(#2,#3, #4
    #1)}} 
\newcommand{\simpleknm}[1][]{\simplepar[#1]{k}{n}{m}}

\title{On the robustness of random $k$-cores}

\author{Cristiane M.~Sato\footnote{The author received an Ontario
    Graduate Scholarship during this project.}\\
  Dept.\ of Combinatorics and Optimization\\
  University of Waterloo\\
  Waterloo ON\\
  Canada N2L 3G1\\
  \small{E-mail: {\tt  cmsato@gmail.com}}
}

\date{\today}

\begin{document}

\maketitle

\begin{abstract}
  The $k$-core of a graph is its maximal subgraph with minimum degree
  at least $k$.  In this paper, we address robustness questions about
  $k$-cores. Given a $k$-core, remove one edge uniformly at random
  and find its new $k$-core. We are interested in how many vertices
  are deleted from the original $k$-core to find the new one. This can
  be seem as a measure of robustness of the original $k$-core. We
  prove that, if the initial $k$-core is chosen uniformly at random
  from the $k$-cores with $n$ vertices and $m$ edges, its
  robustness depends essentially on its average degree~$c$. We prove
  that, if $c\to k$, then the new $k$-core is empty with probability
  $1+o(1)$. We define a constant $c_k'$ such that when $k+\eps< c <
  c_k'-\eps$, the new $k$-core is empty with probability bounded away
  from zero and, if $c > c_k'+ \psi(n)$ with $\psi(n) =
  \omega(n^{-1/4})$, $\psi(n) > 0$ and $c$ is bounded, then
  the probability that the new $k$-core has less than $n-h(n)$
  vertices goes to zero, for every $h(n) =
  \omega(\psi(n)^{-1})$.
\end{abstract}

\section{Introduction}

The $k$-core of a graph is its maximal subgraph with minimum degree at
least~$k$.  The $k$-core of a graph is unique
and it can be obtained by iteratively deleting vertices of degree
smaller than $k$. The $k$-core of a graph that already has minimum
degree at least~$k$ is the graph itself. So we also say that graphs
(and multigraphs) with minimum degree at least~$k$ are $k$-cores.

The investigation of $k$-cores in random graphs was started by
\Bollobas~\cite{Bollobas84} in 1984 in connection with $k$-connected
subgraphs in random graphs. There has been much success in the use of
$k$-cores due to their amenability to analysis. For some earlier
results on the $k$-cores of random graphs,
see~\cite{Luczak91,Luczak92,Molloy92}.

A seminal result in this area was proved by Pittel, Spencer and
Wormald~\cite{PittelSpencerWormald96}: they determined the threshold
$c_k$ for the emergence of a giant $k$-core in $G(n,m)$. Roughly
speaking, if the average degree is below this threshold, the $k$-core
of $G(n,m)$ is empty with probability going to $1$ as $n\to\infty$,
and above the threshold the $k$-core has a linear number of vertices
with probability going to $1$. After this result, many proofs using a
variety of techniques were given for the emergence of a giant $k$-core
in graphs and hypergraphs;
see~\cite{Cooper04,FernholzRamachandaram04,CainWormald06,Kim06,
  JansonLuczak07,JansonLuczak08,Riordan08}.

We are interested in finding how robust this giant $k$-core of
$G(n,m)$ is as a $k$-core. More precisely, if we delete a random edge
in the $k$-core of $G(n,m)$ and obtain its new $k$-core, is the new
$k$-core much smaller than the original one? This can be seen as a
measure of the robustness of the giant $k$-core. We do not restrict
ourselves to the $k$\nobreakdash-core of $G(n,m)$: we consider a $k$-core chosen
uniformly at random with given number of vertices and edges, then we
delete an edge from it uniformly at random and obtain the new
$k$-core. 

We define a constant $c_k'$ and analyse the behaviour of the random
$k$-cores with average degree below and above $c_k'$. We work with
multigraphs with given degree sequence and then we deduce the desired
results for simple graphs. Throughout the paper we use a simple
deletion algorithm (and some variants) to find the $k$-core of a
graph: the algorithm iteratively removes vertices of degree less than
$k$ until all remaining vertices have degree at least $k$. We couple
this deletion algorithm with a random walk. For the case with
bounded average degree $c > c_k'+\psi(n)$ with $\psi(n) =
\omega(n^{-1/4})$ and $\psi(n)>0$, this strategy works quite well: we
prove that the deletion algorithm and the random walk both
terminate/die in less than $t(n)$ steps with probability going to $1$,
for every $t(n) = \omega(\psi(n)^{-1})$. This also implies that, when
$2m/n = c_k+\phi(n) > c_k+n^{-\delta}$, where $\delta \in(0,1/4)$
is a constant, the probability of deleting $\omega(\psi(n)^{-1})$
vertices of the $k$-core of $G(n,m)$ to find its new $k$-core after
deleting a single random edge goes to zero.

For the case with average degree $c\leq c_k'-\eps$ where $\eps$ is a
positive constant, we use the random walk to show that, for any
$h(n)\to\infty$, with probability going to $1$, the deletion algorithm
deletes $\Theta(n)$ vertices or at most $h(n)$ vertices. When $c\to
k$, the probability of deleting $\Theta(n)$ vertices goes to~$1$. Then
we use the differential equation method as described
in~\cite{Wormald99} to show that, if $\Theta(n)$ vertices are deleted,
then the deletion algorithm will not stop until the $k$-core has less
than $\gamma n$ vertices \aas\ (where we can choose $\gamma$ as small
as we want).  Using a result in~\cite{JansonLuczak07}, we prove that
in this case the $k$-core must be empty \aas{} This finishes the proof
that, for $k+\eps \leq c\leq c_k'+\eps$ and any $h(n)\to\infty$, the
deletion algorithm deletes $n$ vertices or at most $h(n)$ vertices
\aas; and that for $c\to k$, we delete $n$ vertices \aas{} Proving that
the probability of deleting all vertices in the case $k+\eps \leq
c\leq c_k'-\eps$ is bounded away from zero require some more work: we
couple the deletion algorithm for multigraphs and simple graphs for
$t(n)\to\infty$ steps. This will then imply that the probability of
deleting $h(n)$ vertices for some $h(n)\to\infty$ is bounded away from
zero and so we must delete all vertices with probability bounded away
from zero.


\section{Main results}

Let $\simple = \simpleknm$ be a graph sampled uniformly at random from
the (simple) $k$-cores with vertex set $[n]$ and $m = m(n)$ edges. For
any graph $H$, let $K(H)$ denote the $k$-core of $H$ and let $W(H)$ be
$\card{V(H)} - \card{V(K( H - e))}$, where $e$ is an edge chosen
uniformly at random from the edges of~$H$. That is, $W(H)$ is the
number of vertices we delete from $H-e$ to obtain its $k$-core. 

For
every $k\geq 0$, let
\begin{equation*}
  f_k(\lambda) = e^\lambda - \sum_{i=0}^{k-1}
  \frac{\lambda^i}{i!}
  \quad\text{and}\quad
  h_k(\mu) = \frac{e^\mu \mu}{f_{k-1}(\mu)}. 
\end{equation*}
 For
$k\geq 3$, let $c_k = \inf\set{h_k(\mu) \st \mu > 0}$ and let
$\mu_{k,c_k}$ be such that $c_k = h_k(\mu_{k,c_k})$. We discuss the
existence of $c_k$ and $\mu_{k,c_k}$ later. Let
\begin{equation*}
  c_k' = \frac{\mu_{k, c_k}
  f_{k-1}(\mu_{k, c_k})}{f_k(\mu_{k, c_k})}.
\end{equation*}

Throughout the text, let $c = 2m/n$. The asymptotics will always be
with respect to $n\to \infty$. For a sequence of probability spaces
$(\Omega_n, \mathbb{P}_n)_{n\in\setN}$, we say that a sequence of
events $(E_n)_{n\in\setN}$ holds asymptotically almost surely (\aas)
if $\mathbb{P}_n(E_n) \to 1$ as $n\to\infty$.

\begin{thm}
\label{thm:main}
Let $k\geq 3$ be a fixed integer. Let $m = m(n)$ and $c = 2m/n$. Then
the following hold.
\begin{itemize}
\item[(i)] If $c\geq k$ and $c\to k$, then $W(\simpleknm) = n$ \aas
\item[(ii)] Let $\eps > 0$ be a fixed real. Suppose that $k+\eps \leq
  c\leq c_k'-\eps$. For any function $h(n)\to \infty$, we have that
  \aas\ $W(\simpleknm) \leq h(n)$ or $W(\simpleknm) = n$. Moreover,
  $W(\simpleknm) = n$ with probability bounded away from zero.
\item[(iii)] Let $\psi(n) = \omega(n^{-1/4})$ be a positive function
  and let $C_0$ be a constant.  Suppose that $c_k'+\psi(n)\leq c\leq
  C_0$. For every $h(n) = \omega(\psi(n)^{-1})$, we have that
  $\prob{W(\simpleknm)\geq h(n)}\to 0$.
\end{itemize}
\end{thm}

We apply Theorem~\ref{thm:main} to study the robustness of the
$k$-core of $G(n,m)$, the random graph chosen uniformly at random
from all graphs on $[n]$ with $m$ edges.
\begin{cor}
  \label{cor:Gnm}
  Let $k\geq 3$ be a fixed integer.  Let $m = m(n)$ and suppose that
  $c = 2m/n = c_k+\psi(n) \geq c_k+ n^{-\delta}$ and $c\leq C_0$,
  where $\delta$ is a constant in $(0,1/4)$ and $C_0$ is a
  constant. Then, for every $h(n) = \omega(\psi(n)^{-1})$, we have
  that $\prob{W\paren[\big]{K\paren[\big]{G(n,m)}} \geq h(n)}\to 0$.
\end{cor}

We remark that there are some known results about the $k$-core of
random graphs with given degree sequence under some constraints on the
degree sequences
(see~\cite{JansonLuczak07,Cooper04,FernholzRamachandaram04}). Since
the degree sequence of a graph $G$ and the degree sequence of $G-e$
for some edge $e\in E(G)$ are very similar, it is intuitive that one
can draw some conclusions about $W(G(k, n, m))$. Indeed, in the case
$c\in [c_k+\eps, C_0]$ one can use~\cite{JansonLuczak07} to conclude
that $W(G(k,n,m)) = o(n)$ \aas{} We were not able to derive results for
the cases (i) and (ii) directly from known results.


\subsection{Models of random multigraphs}

We use the allocation model restricted to $k$-cores (here we allow
multigraphs): let $a: [2m] \to [n]$ be chosen uniformly at random
among the functions such that $|a^{-1}(v)| \geq k$ for any $v\in [n]$;
let $\multi = \multiknm$ be the multigraph on $[n]$ obtained by adding
an edge joining $a(i)$ and $a(m+i)$ for every $i\in [m]$. Then every
simple $k$-core with $n$ vertices and $m$ edges is generated by $m!
2^m$ allocations. This implies that $\multiknm$ conditioned upon
simple graphs is a uniform probability space on $k$-cores with vertex
set $[n]$ and $m$ edges. Multigraphs do not necessarily have the same
probability in $\multiknm$.

Let $\Dknm$ be the set of $\ds\in\setN^n$ with $\sum_{i=1}^n d_i = 2m$
and $\min_i d_i \geq k$. For every multigraph $H$ with vertex set
$[n']$, let $\ds(H)$ denote the degree sequence of $H$, that is,
$(\ds(H))_i$ is the degree of vertex $i$. For any $\ds=(d_1,\dotsc,
d_n)\in \setN^n$, let $D_j(\ds)$ be the number of occurrences of $j$
in $\ds$ and let $\eta(\ds) = \sum_{i=1}^n \binom{d_i}{2}/m$. We will
work with $k$-cores generated using the pairing model with degree
sequences in $\Dknm$. Given a degree sequence $\ds$, let $\multi(\ds)$
denote the graph generated using the pairing model: arbitrarily choose
a partition of $[2m]$ into sets $S_1,\dotsc, S_n$ (which we call bins)
such that $|S_i| =d_i$ for very $i$, add a perfect matching uniformly
at random on $[2m]$ and contract each $S_i$ to obtain a
multigraph. Then $\multiknm$ conditioned upon $\ds(\multiknm)=\ds$ has
the same distribution as $\multi(\ds)$.

It is clear that $\ds(\multi)$ has multinomial distribution
conditioned upon each coordinate being at least $k$, which we denote
by $\multinomialknm$. We say that a variable $Y$ taking integer values
has truncated Poisson distribution with parameters $(k,\lambda)$
(which we denote by $\tpoisson{k}{\lambda}$) if, for every integer
$j$,
\begin{equation*}
  \prob{Y = j}
  =
  \begin{cases}
    {\displaystyle \frac{\lambda^j}{j!f_k(\lambda)}},& \text{ if }j\geq k;\\
    0,&\text{ otherwise.}
  \end{cases}
\end{equation*}
By straightforward computations, one can show that $\multinomialknm$
has the same distribution as $\Ys = (Y_1,\dotsc, Y_n)$ where the
$Y_i$'s are independent truncated Poisson variables with parameters
$(k, \lambda)$ conditioned upon the event $\Sigma$ that $\sum_{i=1}^n Y_i =
2m$.


\section{Random walks and a deletion procedure}

\subsection{A deletion procedure}
\label{sec:exploration_multi_super}
We are given a degree sequence $\ds\in\Dknm$. Here we describe a
procedure for finding the $k$-core of $\multi(\ds)-e$, where $e$ is a
random edge in $\multi(\ds)$. We will sample $\multi(\ds)$ using the
pairing model by discovering one edge at a time. We start by choosing
$e$ by picking two points uniformly at random from the set of all
points.

\medskip

\noindent\textbf{Deletion procedure $(\ds)$}
  \begin{itemize}
  \item Partition $[2m]$ into $n$ bins $S_1,\dotsc, S_n$ such that
    $|S_i|=d_i$ for every $1\leq i\leq n$.
  \item Iteration $0$: Choose $e$ by picking distinct points $u$ and
    $v$ uniformly at random from $[2m]$. Delete $u$ and $v$ and
    \textbf{mark} all points in bins of size less than $k$.
  \item Loop: While there is a marked undeleted point, choose one such point
    $u$ and find the other end $v$ of the edge incident to $u$. Delete
    $u$ and $v$. If $v$ was in a bin of size exactly $k$ (now of size
    $k-1$ because we deleted $v$), mark all the
    other points in this bin.
  \end{itemize}

  After the deletion procedure is over, the $k$-core can be obtained
  by adding a random matching uniformly at random on the surviving
  points. Let $Z_0(\ds)$ denote the number of marked points after the
  deletion of the edge $e$ chosen in Iteration~$0$. Note that
  $Z_0(\ds)\in\set{0, k-2,k-1,2(k-1)}$.

  Let $Y_j(\ds)$ be the number of undeleted marked points after the
  $j$-th iteration of the loop (and $Y_0(\ds) := Z_0(\ds)$). The
procedure stops when $Y_j(\ds) = 0$.  Let $Z_j(\ds)$ be the number of
points that are marked in the $j$-th iteration of the loop. Let
$W(\ds) = \sum_{j}\ceil[\big]{\frac{Z_j(\ds)}{k-1}}$. Note that
$W(\ds) = W(\multi(\ds))$.

We mark new points in an iteration of the loop if $v$ lies in a bin of
(current) size~$k$. The probability that this happens (denoted by
$p_j(\ds)$) is the ratio of the number of unmarked points in bins of
(current) size $k$ and the number of undeleted points other than the
one we are exploring.  If $v$ is also a marked point, then no new
points will be marked and $v$ is deleted. In this case, $Z_j(\ds) =
-1$ and the probability that this happens (denoted by $p_j'(\ds)$) is
the ratio of the marked undeleted points other than $u$ and the number
of undeleted points other than $u$. Thus, in the $j$-th iteration of
the loop,
 \begin{equation*}
   Z_j(\ds)
   =
   \begin{cases}
     k-1,&\text{ with probability } p_j(\ds);\\
     -1,& \text{ with probability } p_j'(\ds)\\
     0,&\text{otherwise}.
   \end{cases}
 \end{equation*}
 The probabilities of $p_j(\ds)$ and $p_j'(\ds)$ are analyzed later.
 
\subsection{Random walks}
Given $c$ and $k$, we will define random walks in $\setZ$ that will
help us to study the behaviour of the deletion procedure. Let
$\lambda_{k,c}$ be the (unique) positive root of $\lambda
f_{k-1}(\lambda) /f_{k}(\lambda) = c$. Such root always exists for $c
> k$. For more properties of $\lambda_{k, c}$,
see~\cite{PittelWormald03}. Let
\begin{equation*} q_{k,c} =
  \frac{\lambda_{k,c}^{k-1}}{(k-1)!f_{k-1}(\lambda_{k,c})}.
\end{equation*}

Let $Z(k,c)$ be a random variable such that
\begin{align*}
  Z(k,c)
  &=
  \begin{cases}
    k-1,&\text{with probability }  q_{k,c};\\
    0,&\text{otherwise.}
  \end{cases}
\end{align*}
Let $Y_0 = Z_0(\ds)$. For $j>0$, let $Y_j = Y_{j-1}+ Z_j-1$ where
$Z_j$ has same distribution as $Z(k,c)$ and the variable $Z_j$ is
independent from $Z_1,Z_2,\dotsc, Z_{j-1}$. Thus, we defined a random
walk such that the position in iteration $j$ is $Y_j$ and the drift is
given by $Z_j-1$. Similarly, for $\xi =\xi(n) \geq 0$ and $\xi \leq
1-q_{k,c}$, define the random variable $Z^+(k,c,\xi)$ by
\begin{align*}
  Z^+(k,c)
  &=
  \begin{cases}
    k-1,&\text{with probability }  q_{k,c}+\xi;\\
    0,&\text{otherwise.}
  \end{cases}
\end{align*}
Let $Y_0^+ = Z_0(\ds)$. For $j>0$, let $Y_j^+ = Y_{j-1}^+ +
Z_j^+-1$ where $Z_j^+$ has same distribution as $Z^+(k,c,\xi)$ and the
variable $Z_j^+$ is independent from $Z_1^+,Z_2^+,\dotsc, Z_{j-1}^+$.
Note that $(Y_j)_{j\in\setN}$ and $(Y_j^+)_{j\in\setN}$ are actually
branching processes.

For $\xi =\xi(n) \geq 0$ and $\xi \leq q_{k,c}$, define
the random variable and $Z^-(k,c,\xi)$ by
\begin{align*}
  Z^-(k,c)
  &=
  \begin{cases}
    k-1,&\text{with probability }  q_{k,c}-\xi;\\
    -1, &\text{with probability }  \xi;\\
    0,&\text{otherwise.}
  \end{cases}
\end{align*}
Let $Y_0^- = Z_0(\ds)$. For $j>0$, let $Y_j^- = Y_{j-1}^- +
Z_j^- -1$ where $Z_j^-$ has same distribution as $Z^-(k,c,\xi)$ and
the variable $Z_j^-$ is independent from $Z_1^-,Z_2^-,\dotsc,
Z_{j-1}^-$.

We say that $Y_j$ is the number of particles alive in iteration $j$
and that $Z_j$ is the number of particles born in iteration $j$ (and
similarly for $Y_j^+$, $Z_j^+$, and $Y_j^-$, $Z_j^-$).

The random walk given by $Z^+(k,c,\xi)$ is going to be used to bound
the number of marked points in the deletion process by above, while
the random walk given by $Z^-(k,c,\xi)$ will bound it from below.
Here we will prove some properties of these random walks.

Recall that $h_k(\mu) = \mu e^\mu/f_{k-1}(\mu)$ and $c_k =
\inf\set{h_k(\mu):\mu > 0)} = h_k(\mu_{k,c_k})$. Here we justify why
the infimum is reached and why it is reached by a unique $\mu$. It is
easy to see that $h_k$ is differentiable. Moreover, $h_k(\mu)\to\infty$ when
$\mu \to 0$ and when $\mu \to \infty$. The first derivative of
$h_k(\mu)$ is
\begin{equation*}
  \frac{e^\mu}{f_{k-1}(\mu)}\paren[\Big]{1+\mu-
    \mu\frac{f_{k-2}(\mu)}{f_{k-1}(\mu)}}
\end{equation*}
Using the fact that $f_{k-2}(\mu) = f_{k-1}(\mu) + \mu^{k-2}/(k-2)!$,
it is clear that this derivative is at least $0$ iff
\begin{equation}
  \label{eq:min_req}
 \frac{\mu^{k-1}}{(k-2)!} \leq f_{k-1}(\mu)
\end{equation}
and the functions on both sides are convex and increasing for $\mu >
0$.  Thus, the function $h_k(\mu)$ must reaches its infimum in a
unique point $\mu_{k,c_k}$ and the equation $h_k(\mu) = c$ has exactly
two roots when $c > c_k$. Let $\mu_{k,c}$ denote the largest root of
the equation $h_k(\mu) = c$. Recall that $c_k' = h_k(\mu_{k,c_k})$.

\begin{prop}
  \label{prop:mean}
  The following hold:
  \begin{itemize}
  \item[(i)] $\esp{Z(k,c)}$ is a strictly decreasing function of $c$
    for $c > k$ and $\esp{Z(k, c_k')} = 1$.
  \item[(ii)] For any $\eps > 0$ with $c_{k}'-\eps > k$, there exists
    a positive constant $\alpha$ such that $\esp{Z(k,c_k'-\eps)} >
    1+\alpha$.
  \item[(iii)] Let $\psi(n)$
  be a nonnegative function with $\psi(n) \leq C_0$, where $C_0$ is
  constant. There exists a positive constant $\beta$ such that
  $\esp{Z(k,c_k'+\psi(n))} \leq 1-\beta\psi(n)$.
  \end{itemize}     
\end{prop}
\begin{proof}
  Let $g(c) = \esp{Z(k,c)}$. Note that $g(c) = (k-1) q_{k,c}$.  By the
  definition of $c_k'$, we have that~\eqref{eq:min_req} holds with
  equality for $\mu = \mu_{k, c_k}$.  This clearly implies $g(c_k') =
  1$. We have that $\lambda_{k,c}$ is a strictly increasing function of
  $c$ and vice-versa (see the derivative computation in~\cite[Lemma
  1]{PittelWormald03}). If $c > k$, then $\lambda_{k,c} > 0$. Thus, by
  considering $c = c(\lambda) = \lambda
  f_{k-1}(\lambda)/f_{k}(\lambda)$ and differentiating with respect
  to $\lambda$, we get
  \begin{equation*}
    \begin{split}
      \frac{d}{d\lambda}q_{k, c}
      &=
      \frac{\lambda^{k-2}
        \left(
          k-1
          -
          \esp{\tpoisson{k-1}{\lambda}}
        \right)}{(k-2)!f_{k-1}(\lambda)}
      < 0
    \end{split}
  \end{equation*}
  since $\esp{\tpoisson{k-1}{\lambda_{k,c}}} > k-1$. Thus, $g(c)$ is
  strictly decreasing for $c > k$. 

  It is easy to see that $c(\lambda)$ is a smooth function on $\lambda
  \in [\eps',\infty)$ for any $\eps' >0$. By the Inverse Function
  Theorem, this implies that $\lambda_{k,c}$ is a smooth function on
  $c\in[c(\eps'),C_0]$ and so $g(c)$ is a smooth function on
  $c$. Thus, the supremum $\sup\set{g'(c): c_k'\leq c\leq C_0}$ and
  the infimum $\inf\set{g'(c): c_k'\leq c\leq C_0}$ are both achieved
  and are both negative constants since $g(c)$ is strictly decreasing. By
  the Mean Value Theorem, there are positive constants $\alpha$ and
  $\beta$ such that $g(c) \geq 1 + \alpha|c-c_k'|$ for $c(\eps') < c < c_k'$ and
  $g(c) \leq 1 -\beta|c-c_k'|$ for $c_k' < c < C_0$.
 \end{proof}

\begin{prop}
  \label{prop:meangreaterthan1}
  Let $k,c,\xi$ be such that $\esp{Z^- (k,c,\xi)} > 1+ \eps$, for some
  constant $\eps > 0$.  Then $\prob{Y_j^- > 0,\ \forall j\geq 0}$ is
  bounded away from $0$ and, for any function $h(n)\to \infty$,
  \begin{equation*}
    \prob[\Big]{Y_j^- > 0,\ \forall j\geq h(n)}
    = 1 + o(1).
  \end{equation*}
\end{prop}
\begin{proof}
  The first part follows from the fact that $(Y_j^-)_{j\geq 0}$ is a
  random walk in $\setR$ with positive expected drift (see
  e.g.~\cite[p.~366]{Feller1}). The second part is a straightforward
  application of the method of bounded differences since the variables
  $Z_j^-$ are independent random variables with range
  $[-1,k-1]$ (see~\cite{McDiarmid89}). 
 \end{proof}


\section{The case $c > c_k'+\omega(n^{-1/4})$}

Here we prove Theorem~\ref{thm:main}(iii). We start by proving a version
of Theorem~\ref{thm:main}(iii) for random multigraphs with given
degree sequence.
\begin{thm}
  \label{thm:c_greater_ck_prime_degrees}
  Let $\psi(n) = \omega(n^{-1/4})$ be a positive function and let
  $C_0$ be a constant.  Suppose that $m = m(n)$ is such that $c =
  2m/n$ satisfies $c_k' + \psi(n)\leq c \leq C_0$. Let $\ds
  \in\Dknm$ be such that $|D_k(\ds) - \esp{D_k(\Ys)}| \leq n \phi(n)$
  for $\phi(n) = o(\psi(n))$, where $\Ys = (Y_1,\dotsc, Y_n)$ and the
  $Y_i$'s are independent truncated Poisson variables with parameters
  $(k, \lambda_{k,c})$. For every $h(n) = \omega(\psi(n)^{-1})$, we have
  that $\prob{W(\multi(\ds)) \geq h(n)} = o(1)$.
\end{thm}

Using Theorem~\ref{thm:c_greater_ck_prime_degrees}, we can deduce a
result about multigraphs with given number of vertices and edges,
which is then used to prove Theorem~\ref{thm:main}(iii).
\begin{cor}
  \label{cor:coupling_multi}
  Let $\psi(n) = \omega(n^{-1/4})$ be a positive function and let
  $C_0$ be a constant.  Suppose that $c = 2m/n$ is such that $c_k' +
  \psi(n)\leq c \leq C_0$.  For every $h(n) = \omega(\psi(n)^{-1})$,
  we have that $\prob{W(\multiknm) \geq h(n)} = o(1)$.
\end{cor}

Now we prove Theorem~\ref{thm:c_greater_ck_prime_degrees}.  We will
choose $\xi$ big enough so that $Z_j(\ds)$ is stochastically bounded
from above by $Z_j^+$ for $j\leq t(n)$ steps, where $Z_j^+$ has the same distribution as $Z^+(k,c, \xi)$. Recall we start the
deletion process with $n$ bins with $d_i$ points inside each bin
$i$. Let $p$ denote the initial ratio between the number of points in
bins of size $k$ and the total number of points.  Note that $p =
kD_k(\ds)/2m = q_{k,c}(1+\phi_1(n))$, for some function $\phi_1(n)$
such that $\phi_1(n) = O(\phi(n))$. Choose $t(n) = \psi(n)^{-1}
n^\alpha$, where $\alpha$ is constant in $(0,1/2)$. Then, for $1\leq
j\leq t(n)$,
\begin{equation*} 
  \frac{kD_k(\ds)-(j+2)(k-1)}{2m-2j-2} \leq p_j(\ds)\leq\frac{kD_k(\ds)}{2m-2j-2}
\end{equation*}
and so $p_j(\ds) = p + O(t(n)/n)$.  We can assume $h(n) \leq
t(n)$. Proposition~\ref{prop:mean} implies that $q_{k,c}\leq 1/(k-1)$.
Since $t(n)/n = o(\psi(n))$, we can choose $\xi > 0$ such that $\xi =
o(\psi(n))$ and $\xi < 1-q_{k,c}$ and $Z^+_j \geq Z_j$ for all $j\leq
t$. Now $\esp{Z^+(k,c,\xi)} \leq 1 - \beta\psi(n) + (k-1)\xi$ according
to Proposition~\ref{prop:mean} for some positive constant
$\beta$. Since $\xi = o(\psi)$, we have $\esp{Z^+(k,c,\xi)} \leq 1 -
\beta'\psi(n)$ for some positive constant $\beta'$. Thus, we have that
$\esp{Y_{t}^+} = O((1-\beta'\psi(n))^{t(n)}) =
O(\exp(-t(n)\beta'\psi(n) )) = o(1)$ because $t(n) =
n^{\alpha}/\psi(n)$ with $\alpha > 0$.  This implies that the deletion
procedure stops before $t(n)$ steps \aas, which proves
Lemma~\ref{thm:c_greater_ck_prime_degrees}.

\subsection{Proof of Corollary~\ref{cor:coupling_multi} and
  Theorem~\ref{thm:main}(iii)}
\label{sec:proofcor_super}
Let $h(n)=\omega(\psi(n)^{-1})$. Choose $\phi(n)$ such that $\phi(n) =
o(\psi(n))$ and $\phi(n) = \omega(n^{-1/4})$. First we will prove
Corollary~\ref{cor:coupling_multi}. We will show that the degree
sequences that satisfy the hypotheses in
Lemma~\ref{thm:c_greater_ck_prime_degrees} are the `typical' degree
sequences for $\multiknm$. Let $\tDknm$ be the set of degree sequences
$\ds$ satisfying $|D_k(\ds) - \esp{D_k(\Ys)}|\leq n\phi(n)$. Recall that
$\ds(\multiknm)$ has the same distribution as $\Ys = (Y_1,\dotsc,
Y_n)$ such the $Y_i$'s are independent truncated Poisson variables
with parameters $(k, \lambda_{k,c})$ and conditioned to the event
$\Sigma$ that $\sum_i Y_i = 2m$. Using Chebyshev's inequality,
\begin{equation*}
  \prob{|D_k(\Ys) - \esp{D_k(\Ys)}| \geq \phi(n) n}
  \leq
  \frac{n}{n^2\phi(n)^2}.
\end{equation*}
By~\cite[Theorem~4(a)]{PittelWormald03}, it is easy to see that the
probability of $\Sigma$ is $\Omega(1/\sqrt{n})$. Thus,
\begin{equation}
  \label{eq:probtypicaldegree}
  \prob{\ds(\multiknm) \not\in \tDknm}
  \leq
  \frac{\prob{\Ys \not\in\tDknm}}{\prob{\Sigma}}
  =
  O\left(\frac{n\sqrt{n}}{n^2 \phi(n)^2}\right)
  =o(1).
\end{equation}

For every $n\in\setN$, since the set $\tDknm$ is finite, there exists
a degree sequence $\ds^*(n)$ such that
$\prob[\big]{W(\multi(\ds^*(n))) \geq h(n)}=
\max\set[\big]{\prob[\big]{W(\multi(\ds)) \geq h(n)}: \ds \in
  \tDknm}$.  Set $r(n) =\prob[\big]{W(\multi(\ds^*(n))) \geq
  h(n)}$. Theorem~\ref{thm:c_greater_ck_prime_degrees} implies that
$r(n) = o(1)$. Thus, for any sequence $(\ds(n))_{n\in\setN}$ such that
$\ds(n)\in\tDknm$ for every $n\in\setN$, we have that
$\prob{W(\multi(\ds(n))) \geq h(n)}\leq r(n) = o(1)$. This is usually
expressed by saying that $\prob{W(\multi(\ds(n))) \geq h(n)}\to 0$
uniformly for $\ds\in\tDknm$. Together
with~\eqref{eq:probtypicaldegree}, this implies that
$\prob{W(\multiknm)\geq h(n)} = o(1)$, proving
Corollary~\ref{cor:coupling_multi}.

We will now prove Theorem~\ref{thm:main}(iii).  To deduce the result for
simple graphs, we impose further conditions on the degree sequences:
let $\hDknm$ be the set of degree sequences in $\tDknm$ that satisfy
the conditions that $\max_i d_i \leq n^{\eps}$ for some $\eps
\in(0,0.25)$ and that $|\eta(\ds) - \esp{\eta(\Ys)}| \leq
\phi(n)$. One can easily prove that $\var{Y_i(Y_i-1)} = O(1)$ and so
uniformly for $n$ and~$m$ with $c < C_0$, by Chebyshev's inequality,
\begin{equation*}
\begin{split}
\prob[\Big]{|\eta(\Ys) - \esp{\eta(\Ys)}|
\geq \phi(n)}
=
O\left(
\frac{1}{n\phi(n)^2}
\right).
\end{split}
\end{equation*}
For $j_0 > 2e\lambda_{k,c}$, we have that $\prob{Y_1 > j_0} =
O(\exp(-j_0/2))$. This holds because the ratio
$\prob{Y_1=j+1}/\prob{Y_1=j}$ is less than $1/e$ for $j \geq j_0/2$
(This is the same equation as~\cite[Equation
(27)]{PittelWormald03}). Thus, $\prob{\max_j Y_j \geq n^{\eps}} =
O(n\exp(- n^{\eps}/2))$. This implies that $\prob{\ds(\multiknm)\in
  \hDknm}$ is also $1+o(1)$.  For $\ds\in\hDknm$, the probability of
that $\multi(\ds)$ is simple is already known (see~\cite{McKay85,
  McKayWormald91}):
\begin{equation*}
  \begin{split}
    \prob{\multi(\ds)\text{ simple}}
    &=
    \exp
    \left(
      -\frac{\eta(\ds)}{2}
      -\frac{\eta(\ds)^2}{4}
      +
      O\left(\frac{\max_i d_i^4}{n}\right)
    \right)
    \\
    &\sim
    \exp
    \left(
      -\frac{\bar\eta_{c}}{2}
      -\frac{(\bar\eta_{c})^2}{4}
      +
      O\left(\frac{\max_i d_i^4}{n}\right)
    \right),
  \end{split}
\end{equation*}
where $\bar \eta_c := \lambda_{k,c} f_{k-2}(\lambda_{k,c})/
f_{k-1}(\lambda_{k,c})$. We can apply the same argument on the
uniformity of the bound for $\prob{W(\multi(\ds))\geq h(n)} = o(1)$ as
above to $\prob{\multi(\ds)\text{ is simple}} - \exp\left(-\eta_{c}/2
  -\eta_{c}^2/4\right)$ and conclude that $$\prob{\multiknm \text{ is
    simple}} = \exp\left(-\eta_{c}/2 -\eta_{c}^2/4\right) +o(1) =
\Omega(1)$$ and so
\begin{equation*}
  \begin{split}
    &\prob[\Big]{W(\simpleknm)\geq h(n)}
    \\
    &=
    \probcond[\Big]{W(\multiknm)\geq h(n)}{\multiknm\text{ is simple}}
    \\
    &\leq
    \frac{\prob{W(\multiknm)\geq h(n)}}
    {\prob{\multiknm\text{ is simple}}}
    \\
    &= o(1).
  \end{split}
\end{equation*}
This finishes the proof of Theorem~\ref{thm:main}(iii).

\subsection{The $k$-core of $G(n,m)$}
In this section we prove Corollary~\ref{cor:Gnm}. We will
use~\cite[Theorem 2]{PittelSpencerWormald96}. Although this result
does not state the number of edges in the $k$-core, it can be obtained
from its proof with the main steps in~\cite[Equations
(6.18),(6.34)]{PittelSpencerWormald96} and~\cite[Corollary
1]{PittelSpencerWormald96} applied to $J_1$. We
restate~\cite[Theorem~2]{PittelSpencerWormald96} with the number of
edges here:
\begin{thm}[\protect{\cite[Theorem~2]{PittelSpencerWormald96}}]
  Suppose $c > c_k + n^{-\delta}$, $\delta\in (0,1/2)$ being
  fixed. Fix $\sigma \in (3/4,1-\delta/2)$ and $\bar \zeta =
  \min\set{2\sigma -3/2, 1/6}$. Then with probability $\geq 1 +
  O(\exp(-n^\zeta))$ ($\forall \zeta < \bar\zeta$), the random graph
  $G(n,m=cn/2)$ contains a giant $k$-core with $e^{-\mu_{k,c}}
  f_k(mu_{k,c}) n + O(n^{\sigma})$ vertices and $(1/2) \mu_{k,c}
  e^{-\mu_{k,c}} f_{k-1}(\mu_{k,c})n + O(n^{\sigma})$ edges.
\end{thm}

We are now ready to prove Corollary~\ref{cor:Gnm}. Recall that $c
\geq c_k + n^{-\delta}$, where $\delta \in (0, 1/4)$. So $\delta =
1/4-\eps$, where $\eps$ is a constant in $(0,1/4)$.  Let $\eps' <
\eps$ be a constant such that $\eps'< 1/4-\delta/2$. Fix $\sigma =
3/4+\eps'$. Thus, the average degree of the $k$-core is
\begin{equation*}
  \frac{\mu_{k,c}
    f_{k-1}(\mu_{k,c})}{f_k(\mu_{k,c})}(1+O(n^{-1/4+\eps'}).
\end{equation*}
Recall that $h'(\mu_{k,c_k})=0$ and $h'(\mu) > 0$ for $\mu
>\mu_{k,c_k}$. This implies that $\mu_{k, c} = \mu_{k, c_k} +
\Omega(c-c_k)$. Moreover, the function $x\mapsto xf_{k-1}(x)/ f_k(x)$
is smooth. Thus, the average degree of the $k$-core of $G(n,m)$ is
$(c_k' + \Theta(c-c_k))(1+O(n^{-1/4+\eps'}))$. Since $c-c_k' >
n^{-\delta} = n^{-1/4+\eps}$ with $\eps > \eps'$, the average degree
of the $k$-core is $c_k'+\Omega(c-c_k)$. We can now apply
Theorem~\ref{thm:main}(iii) to obtain the desired result.


\section{The case $k\leq c \leq c_k'-\eps$: deleting $\Theta(n)$
  vertices}
\label{sec:inter}
The following result is an intermediate step for the proof of
Theorem~\ref{thm:main}(i) and (ii) .
\begin{thm}
  \label{thm:inter_degrees}
  Let $\eps > 0$ be a fixed real.  Suppose that $k\leq c \leq
  c_k'-\eps$.  Let $\phi(n) = o(1)$. Let $\ds$ be such that $D_k(\ds)
  \geq \esp{D_k(\Ys)}(1 - \phi(n))$, where $\Ys = (Y_1,\dotsc, Y_n)$
  and the $Y_i$'s are independent truncated Poisson variables with parameters
  $(k, \lambda_{k,c})$. Then there exists a constant $\eps' > 0$
  (depending on $\eps$) such that, for every function $h(n)\to
  \infty$, we have that \aas\ $W(\multi(\ds)) \leq h(n)$ or $W(\multi(\ds))
  \geq \eps' n$. Moreover, $W(\multi(\ds)) \geq \eps' n$ with
  probability bounded away from zero.
\end{thm}

The proof of the following corollary is very similar to the proof
in Section~\ref{sec:proofcor_super} and so we omit it.
\begin{cor}
  \label{cor:multi_inter}
  Let $\eps > 0$ be a fixed real.  Suppose that $k\leq c \leq
  c_k'-\eps$. Then there exists a constant $\eps' > 0$ (depending on
  $\eps$) such that, for every function $h(n)\to \infty$, we have that
  \aas\ $W(\multiknm) \leq h(n)$ or $W(\multiknm) \geq \eps' n$.
  Moreover, $W(\multiknm)) \geq \eps' n$ with probability bounded away
  from zero.
\end{cor}

For the case $c\to k$, Theorem~\ref{thm:inter_degrees} implies a
stronger result because there is a function $h(n)\to \infty$ such that
$W(\multiknm)\geq h(n)$ steps \aas\ From this
one can deduce the following result.
\begin{cor}
  \label{cor:goingtok}
  If $c\geq k$ and $c\to k$, then there exists a constant $\eps'> 0$
  such that $W(\multiknm)) \geq \eps' n$ \aas
\end{cor}

Now we prove Theorem~\ref{thm:inter_degrees}.  We will choose $\xi$ so
that $Z_j(\ds)$ is stochastically bounded from below by $Z_j^-$, where
$Z_j^-$ has the same distribution as $Z^-(k,c,\xi)$, and so that
$\esp{Z^-(k,c,\xi)}$ is bounded away from $1$ from above. For $j\geq
1$, we have already seen that $p_j(\ds) =
(1+O(j/n)+O(\phi(n)))q_{k,c}$. Moreover, for $j\geq 1$, the
probability that $Z_{j}(\ds)=-1$ is at most $(k-1)(j+2) /(2m-2j-2)$.

Proposition~\ref{prop:mean} implies that $q_{k,c} > 1/(k-1)$ and that
$\esp{Z^-(k,c,\xi)}\geq 1 + \alpha' -(k-1)\xi$ for some constant
$\alpha'>0$. Choose $\xi \in (0,\alpha'/(k-1))$. Thus, we have
$\esp{Z^-(k,c,\xi)}\geq 1 +\alpha$ for some $\alpha > 0$.  We can now
choose $\eps'' > 0$ small enough so that $p_j(\ds) \geq q_{k,c} -\xi$
and $\prob{Z_j(\ds)=-1}\leq \xi$ for all $j\leq \eps''n$. Thus, we can
couple the processes for at least $t(n) = \eps''n$ steps.

By Proposition~\ref{prop:meangreaterthan1}, \aas\ either $Y_j^- \leq
0$ for some $j\leq h(n)$ or $Y_{j}^- > 0$ for all~$j$.  Moreover,
the latter occurs with probability bounded away from zero. Since the
coupling holds for $t(n)$ steps with $Z_j^-\leq Z_{j}(\ds)$, \aas\
either $Y_j(\ds) = 0$ for some $j\leq h(n)$ or $Y_{j}(\ds) > 0$ for
$1\leq j\leq \eps''n$. Thus, \aas\ either $W(\multi(\ds))\leq
h(n)+2$ or $W(\multiknm) \geq \eps''n/(k-1)$. This completes the proof
of Theorem~\ref{thm:inter_degrees}.

\section{The case $k\leq c \leq c_k'-\eps$}
In this section we will prove Theorem~\ref{thm:main}(i). We will also
prove Theorem~\ref{thm:main}(ii) except for the claim that
$W(\simpleknm) = n$ with probability bounded away from zero, which is
handled in Section~\ref{sec:simple}. We use
the differential equation method as described in~\cite[Theorem
6.1]{Wormald99} with stopping times. We will also use some results
from~\cite{CainWormald06}.

We will use the pairing-allocation model $\Pcal(M,L, V, k)$ as
described in~\cite{CainWormald06}: given a set $M$ of points together
with a perfect matching $E_M$ on $M$ and two disjoint set $L, V$ let
$h$ be chosen uniformly at random from the functions mapping $M$ to
$L\cup V$ such that $|h^{-1}(v)| \geq k$ for all $v\in V$ and
$|h^{-1}(v)|=1$ for all $v\in L$. Let
$G_{\Pcal} = G_{\Pcal}(M,L,V,k)$ be the multigraph obtained by adding
edges joining $h(a)$ and $h(b)$ for every $ab\in E_M$ and $h(a),
h(b)\in V$. Note that $\multiknm = G_{\Pcal}([2m],\varnothing, [n],k)$
with $E_M = \set{\set{i, m+i}: i\in[m]}$.

We say that the vertices in $V$ are heavy vertices and the vertices in
$L$ are light vertices. We will also say that point $i\in M$ is in $v$
if $h(i) = v$.

Cain and Wormald~\cite{CainWormald06} analyse a deletion procedure for
obtaining the $k$-core. Here we will use a similar procedure with the
only modifications being in the first step. The procedure receives as
input $h:[2m]\to [n]$ such that $|h^{-1}(v)| \geq k$ for all $v\in
[n]$.

\medskip

\noindent \textbf{Deletion procedure -- pairing--allocation $(h)$:}
\begin{itemize}
\item Let $M = [2m]$, $L = \varnothing$ and $V = [n]$.
\item Iteration 0: Choose $i\in [m]$ uniformly at random. Find $v =
  h(i)$ and $u = h(m+i)$. Delete $i$ and $m+i$ from $M$.  If $u\neq v$
  and $|h^{-1}(v)|= k$, then delete $v$ from~$V$, add $k-1$ new
  elements to $L$ and redefine the action of $h$ on
  $h^{-1}(v)\setminus\set{i}$ as a bijection to the new
  elements. Similarly to~$u$, if $u\neq v$ and $|h^{-1}(u)|= k$, then
  delete $u$ from~$V$, add $k-1$ new elements to $L$ and redefine the
  action of $h$ on $h^{-1}(u)\setminus\set{m+i}$ as a bijection to the
  new elements. If $u=v$ and $|h^{-1}(v)|\leq k+1$, then delete $v$
  from $V$, add $|h^{-1}(v)|-2$ new elements to $L$ and redefine the
  action of $h$ on $h^{-1}(v)\setminus\set{i, m+i}$ as a bijection to
  the new elements.
\item Loop: While $L\neq\varnothing$, choose $j\in h^{-1}(L)$ uniformly at
  random. Delete $j$ and $m+j$ of $M$ and delete $h(j)$ from $L$. Find
  $v= h(m+j)$. If $v \in L$, delete $v$ from $L$. If $v\in V$ and
  $|h^{-1}(v)|=k$, then delete $v$ from $V$, add $k-1$ new elements to
  $L$ and redefine the action of $h$ on $h^{-1}(v)\setminus\set{i}$ as
  a bijection to the new elements.
\end{itemize}
Let $h_0, M_0. L_0, V_0$ be the values of $h, M, L, V$, resp., after
Iteration $0$. Let $h_i, M_i, L_i, V_i$ be the values of $h, M, L, V$,
resp., after the $i$-th iteration of the loop. Then the proof
of~\cite[Lemma~6]{CainWormald06} gives us the same conclusion
as~\cite[Lemma 6]{CainWormald06}:
\begin{lem}
  Starting with $h = \Pcal([2m], \varnothing, [n], k)$ and
  conditioning upon the values of $M_i, L_i$ and $V_i$, we have that
  $h_i$ has same distribution as $\Pcal(M_i, V_i, L_i,k)$.
\end{lem}

\subsection{The case $c\to k$}
Here we prove Theorem~\ref{thm:main}(i). We assume that $c = 2m/n = k+
\phi(n)$, where $\phi(n) = o(1)$ and $\phi(n) \geq 0$.  Let $S_i$
denote the number of points in heavy vertices just after the $i$-th
iteration of the loop. Let $S_0$ denote the number of points in heavy
vertices after Iteration~$0$.  We will use $x$ as $i/n$ and $y(i/n)$
to approximate $S_i/n$.

Define $D_\gamma = \set{(x,y)\st -\gamma < 2x < k -\gamma,\ \gamma < y
  < k + \gamma}.$ Note that $D_\gamma$ is bounded, connected and
open. We choose $\gamma < \min\set{\gamma_0/3, k}$ so that the
$k$-core cannot not empty and smaller than $\gamma_0 n$ \aas
($\gamma_0$ is given by Lemma~\ref{lem:small}). Moreover, we work with
$n$ big enough so that $\phi(n) < \gamma$.  After the first step there
are at most $2(k-1)$ points in $L_0$ and all the other vertices in
$V_0$ and so $S_0 \geq 2m-2(k-1)$. Then it is clear
that $S_0/n \leq k + \phi < k + \gamma$ and $S_0/n > \gamma$.

Let $T_D = \min\set{i: (i/n, S_i/n)\not\in D}$. Let $W_i$ denote the
number of light vertices after iteration $i$ is performed.  We also use
the stopping time $T = \min\set{i: W_i = 0}$. That is, there are no
light vertices to be deleted and the deletion process has actually
ended.  We need to check the boundedness hypothesis, trend hypothesis
and Lipschitz hypothesis (see~\cite{Wormald99} for more details). The
boundedness hypothesis is trivially true: $|S_i - S_{i+1}|\leq k$
always.

Now we check the trend hypothesis. Let $f(x,y) = -ky/(k-2x)$.  Let
$H_i$ denote the history of the process at iteration $i\geq 1$. We
need to show that $\xi_1 :=| \esp{S_{i+1}-S_i | H_{i}}-f(i/n,S_i/n)| =
o(1)$ while the $i < T$ and $i < T_D$. We have that $S_{i+1}-S_i$ is
zero if $j$ is matched to a light vertex, is $-1$ if $j$ is matched to
a point in a heavy vertex with degree $> k$ and is $-k$ if $j$ is
matched to a point in a heavy vertex with degree exactly $k$.  The
probability that $j$ is matched to a point in a heavy vertex is $S_i/
(2m-2i-2)$. The probability that such a heavy vertex has degree $k$ is
at least $1 - \sum\set{d_i: d_i > k}/ S_i$ where $\ds$ is the degree
sequence (we do not sample the degree sequence, we just
decide if the vertex had degree $k$ or not). But for every possible
degree sequence $1 - \sum\set{d_i: d_i > k}/ S_i \geq 1 + n\phi(n)/S_i
= 1+o(1)$ whenever $S_i \geq \gamma n$.Then
\begin{equation*}
  \esp{S_{i+1}-S_i | H_{i}}
  =
  \frac{-k |S_i|}{2m-2i-2}+o(1),
\end{equation*}
and so the trend hypothesis holds. It is easy to see that the
Lipschitz hypothesis also holds in $D_\gamma \cap \set{(x, y): x\geq 0}$. 

 According to~\cite[Theorem
6.1]{Wormald99}, the $y'(x) = f(x,y)$ has a unique solution in
$D_\gamma$, say $y^*$, with $y(0)=k$ and a unique solution in
$D_\gamma$, say $y^{**}$, with $y(0) = S_0/n$. Note that $y^*$ is a
fixed function while $y^{**}$ is a random variable because $S_0$ is a
random variable. The Lipschitz condition implies that, for any $x$
with both $(x, y^*(x))$ and $(x, y^{**}(x))$ in $D_\gamma$, we have
that $|y^*(x)-y^{**}(x)| = x |k-S_0/n| R =: \xi_3$, where $R$ is some
big constant and so $\xi_3=o(1)$. Let $\xi_2 = o(1)$ and $\xi_2 >
\xi_1$ and $\xi_2 > \xi_3$. By~\cite[Theorem
6.1]{Wormald99}, there is a constant $C$ and a function $\xi \to 0$,
such that, \aas\ at each step $i < \min\set{T,n\sigma}$ we have that
\begin{equation}
  \label{eq:formula}
  |S_i - ny^{*}(i/n)| \leq \xi n,
\end{equation}
where $\sigma$ denotes the supremum of $x$ such that $(x', y^{*}(x'))$
and $(x', y^{**}(x'))$ are at $\ell^{\infty}$-distance at least
$C\xi_2$ of the boundary of $D_{\gamma}$ for all $0\leq x'\leq x$.

Let $\eps'$ be given by Corollary~\ref{cor:goingtok}. For $\eps' < x <
(k-\gamma)/2$, we have that
\begin{equation*}
  (k-2x) - y^*(x)
  =
  (k-2x)
  \left( 1 - \left(\frac{k-2x}{k}\right)^{k/2-1}\right)
  \geq
  \frac{2\gamma\eps'}{k}.
\end{equation*}
This implies that, if~\eqref{eq:formula} holds at $i$ where $\eps'n
<2i< (k-\gamma)n$, then $W_i = 2m-2i-2-S_i = \Omega(n)$. Thus,
if~\eqref{eq:formula} holds for some step $i\in (\eps'n, \sigma n]$
with $T >i$, then $T > i+1$ because there are still $\Omega(n)$ points
to be deleted. This implies that, conditioning upon $T > \eps'n$, we
have that $T > \sigma n$ \aas

For any constant $\alpha \in (0,\gamma)$, using the fact that $\xi_3=
o(1)$, there exists $x$ such that $x\leq \sigma n$ and $(x,y^*(x))$
and $(x,y^{**}(x))$ are at $\ell^{\infty}$-distance $(C\xi_2, \alpha)$
of the boundary of $D_\gamma$. For such an $x$ we have $T > x$ \aas\
because $T > \sigma n$ \aas{} Thus, \eqref{eq:formula}~holds \aas{}
Since $x$ is at $\ell^\infty$-distance at most $\alpha$ of the
boundary of $D_\gamma$, either $2x \geq k - \gamma - \alpha$ or
$y^*(x) \leq \gamma+\alpha$. We excluded $y^*(x) \geq k+\gamma-\alpha$
because $y^*(0) = k$ and $y^*$ decreases as $x$ increases. For $n$
sufficiently large so that $|\xi(n)| < \gamma$, the
equation~\eqref{eq:formula} for at either $2x \geq k - \gamma - \alpha$ or
$y^*(x) \leq \gamma+\alpha$ shows that $S_i\leq n\gamma_0$ \aas

Since $T >\eps'n$ \aas\ by Corollary~\ref{cor:goingtok}, the $k$-core
would have to be smaller than $\gamma_0 n$ and so it must be empty
\aas\ (see Section~\ref{sec:small}).  We conclude that $W(\multiknm) =
n$ \aas{} Since the probability that $\multiknm$ is simple is
$\Omega(1)$ (see~\cite{McKay85, McKayWormald91}), we have that
$W(\simpleknm) = n$ \aas


\subsection{The case $c\in [k+\eps, c+k' - \eps]$} 

\label{sec:de-inter}
We prove Theorem~\ref{thm:main}(ii) except for the claim that
$W(\simpleknm) = n$ with probability bounded away from zero, which is
addressed in Section~\ref{sec:simple}.  Let $h(n)\to \infty$. Let
$\eps'$ be given by Corollary~\ref{cor:multi_inter}.  Assume that $c\to
C \in [k+\eps, c+k' -\eps]$, where $C$ is a constant. We will explain
later how to drop this constraint.

Again we use the differential equation method as in~\cite{Wormald99}
and the deletion procedure described in the beginning of the section.
For each~$i$, let $S_i$ denote the number of points in heavy vertices
just after iteration $i$, let $T_i$ denote the number of heavy
vertices just after iteration $i$ and let $W_i$ denote the number of
points in light vertices just after iteration~$i$. We will use the
differential equation method to approximate $S_i$ and $T_i$. Note that
$W_i = 2m - 2i - 2 - S_i$. We will use $y(i/n)$ to approximate $S_i/n$
and $z(i/n)$ to approximate $T_i/n$.

Let $\gamma$ be a positive constant with $\gamma < \min\set{1, C-k}$
to be chosen later. Define
\begin{equation*}
  D_\gamma = \set{(x,y,z)\st \gamma < z < 1 + \gamma,\ 
    -\gamma < x < C-\gamma,\ 
    \gamma < y < C+\gamma, y > (k+ \gamma)z}.
\end{equation*}
Then $D_\gamma$ is bounded, connected and open. We have $T_0\in
\set{n, n-1, n-2}$ and $S_0 \in[2m-2-2(k-1), 2m-2]$. Thus, $T_0/n = 1 + o(1/n)$ and $S_0/n = C+o(1)$.
Then $D_\gamma$ contains the closure of the points $(0, y,z)$ such
that $\prob{S_i = yn \text{ and }T_i=zn}\neq 0$ for some~$n$.

We use the stopping time $T = \min\set{i: W_i = 0}$ again. We have to
check the boundedness hypothesis, the trend hypothesis and the
Lipschitz hypothesis. The boundedness hypothesis is again easy:
$|S_{i+1}-S_i|\leq k$ and $|T_{i+1}-T_i|\leq 1$ always.

The trend hypothesis is exactly like in~\cite{CainWormald06} with
$|\esp{T_{i+1}-T_i| H_i} -f_z(i/n)| = \xi_z = o(1)$ and
$|\esp{S_{i+1}-S_i| H_i} -f_y(i/n)| = \xi_y = o(1)$ with
\begin{align*}
  f_z(x)
  &=
  - \frac{y}{C-2x}
  \left(1-\frac{\mu z}{y}\right)
  \quad\text{and}\quad f_y(x)
  =
  - \frac{y}{C-2x}
  \left(k-(k-1)\frac{\mu z}{y}\right),  
\end{align*}
where $\mu = \lambda_{k, y/z}$. The Lipschitz hypothesis is
straightforward to check.

According to~\cite[Theorem 6.1]{Wormald99}, the $y'(x) = f_y(x)$ and
$z'(x) = f_z(x)$ has unique solutions $(y^*, z^*)$ and $(y^{**},
z^{**})$, with initial conditions $y(0)=C$ and $z(0)=1$, and $y(0) =
S_0/n$ and $z(0) = T_0$, resp. The Lipschitz hypothesis implies that,
there exists a constant $R$ such that, for any $x$ with both $(x,
y^*(x),z^*(x))$ and $(x, y^{**}(x), z^{**}(x))$ in $D_\gamma$, we have
 $\max\set{|y^*(x)-y^{**}(x)|,|z^*(x)-z^{**}(x)|} \leq x
|k-S_0/n|R =: \xi_3$. Note that $\xi_3=o(1)$.  Let $\xi_2 > \xi_z$,
$\xi_2 > \xi_y$, $\xi_2>\xi_3$, $\xi_2 = o(1)$. Thus, by~\cite[Theorem
6.1]{Wormald99}, there is a constant $C_0$ and a function $\xi \to 0$,
such that, \aas\ at each step $i < \min\set{T,n\sigma}$ we have that
\begin{equation}
  \label{eq:formula2}
  \max\set{|S_i - ny^{*}(i/n)|, |T_i - n z^*(i/n)|} \leq \xi n,
\end{equation}
where $\sigma$ denotes the supremum of $x$ such that, for all $0\leq
x'\leq x$, we have $(x' ,
y^{*}(x'),z^*(x'))$ and $(x', y^{**}(x'),z^{**}(x'))$ are at
$\ell^{\infty}$-distance at least $C_0\xi_2$ of the boundary of
$D_{\gamma}$.

According to~\cite{CainWormald06}, we have that $\mu^2/(C-2x)$ and
$(ze^{\mu})/f_k(\mu)$ are constants. With initial conditions $y(0) =
C$ and $z(0) = 1$, we get $\mu^2/(C-2x)=\lambda_{k,C}^2/C$ and
$ze^{\mu}/f_k(\mu) = e^{\lambda_{k,C}}/f_k(\lambda_{k,C})$, which can be used
to deduce that
\begin{equation}
  \label{eq:diffeq}
  y^* = (C-2x)\frac{h_k(\lambda_{k,C})}{h_k(\mu)}.
\end{equation}

For $x\geq \eps' /2$, we must have $\mu(x) \leq \lambda_{k,C}
\sqrt{1-\eps'/C}$ and so $h_k(\mu)\geq (1+\eps'')h_k(\lambda_{k,C})$,
for some $\eps''>0$. Thus, for every $x$ such that $\eps' \leq 2x \leq
C-\gamma$ using~\eqref{eq:diffeq},
\begin{equation*}
  C-2x - (C-2x)\frac{h_k(\lambda_{k,C})}{h_k(\mu)} \geq \gamma\eps''.
\end{equation*}
This implies that, if~\eqref{eq:formula} holds at $i$ with $\eps'n
<2i<(C-\gamma)n$, then $W_i = 2m-2i-2-S_i = \Omega(n)$. Thus,
if~\eqref{eq:formula} holds for some step $i\in (\eps'n, \sigma n]$
with $T >i$, then $T > i+1$ because there are still $\Omega(n)$ points
to be deleted. This implies that, conditioning upon $T > \eps'n$, we
have that $T > \sigma n$ \aas

For any constant $\alpha \in (0,\gamma)$, using the fact that $\xi_3=
o(1)$, there exists $x$ such that $x\leq \sigma n$ and $(x,y^*(x),
z^*(x))$ and $(x,y^{**}(x), z^{**}(x))$ are at
$\ell^{\infty}$-distance $(C_0\xi_2, \alpha)$ of the boundary of
$D_\gamma$. For such an $x$ we have $T > x$ \aas\ because $T > \sigma
n$ \aas{} Thus, \eqref{eq:formula2}~holds \aas{} Since $x$ is at
$\ell^\infty$-distance at most $\alpha$ of the boundary of $D_\gamma$,
either $z^*(x)\leq \gamma+\alpha$ or $2x\geq C-\gamma-\alpha$ or
$y^*(x) \leq \gamma +\alpha$ or $y^*(x) / z^*(x) \leq k + \gamma
+\alpha$.  We excluded $y^*(x) \geq C+\gamma-\alpha$ and $z^*(x)\geq
1+\gamma-\alpha$ because $y^*(0) = C$ and $z^*(0)=1$ and $f_y$ and
$f_z$ are decreasing.

If $2x \geq C - \gamma - \alpha$, then, using that $\mu^2/(C-2x)$
remains constant and $\mu(0) = \lambda_{k,C}$ with $C < c_k' $, we have
that $h_k(\mu)\geq h_k(\lambda_{k,C})$ and so $y^*(x) \leq C-2x\leq
\gamma+\alpha$. For $n$ sufficiently large so that $|\xi(n)| < \gamma$, it is easy to see that the
equation~\eqref{eq:formula} for at $y^*(x) \leq \gamma+\alpha$ or
$z^*(x)\leq \gamma+\alpha$ implies that $S_i\leq 3\gamma n$ \aas{} We
still have to check what happens when $y^*(x) / z^*(x) \leq k + \gamma
+\alpha$. In this case, $\mu = O(\gamma+\alpha)$. This holds because
the function $g_k(x) = x f_{k-1}(x)/f_k(x)$ with domain $(0, 2\gamma)$
is strictly increasing and the limit of its one-sided derivative with
$x\to 0$ is $1/(k+1)$. This limit can be computed by the derivative of
this function which is $(1/x)g_k(x) (1+g_{k-1}(x)-g_k(x))$ and then
using Taylor's approximation for $g_k(x)$ and $g_{k-1}(x)$ around
$x\to 0$. For more details on this computation see~\cite[Lemma
1]{PittelWormald03} and its proof. Using that $\mu^2/(C-2x)$ during
the process, we then have $C-2x = O(\gamma^2)$, we can then conclude
that the $S_i = O(\gamma n)$ \aas

Thus, conditioned upon $T > \eps'n$ the $k$-core has at most $O(\gamma
n)$ vertices \aas{} Let $\gamma_0$ be the constant given by
Lemma~\ref{lem:small}.  By choosing $\gamma$ small enough, we can
conclude that, conditioned upon $T > \eps'n$, the $k$-core has less than
$\gamma_0 n$ vertices \aas\ and which implies, by
Lemma~\ref{lem:small}, that the $k$-core must be empty \aas{} By
Corollary~\ref{cor:multi_inter}, we have that $W(\multiknm)\leq h(n)$
or $W(\multiknm)=n$ with probability $1+o(1)$ conditioned upon $T >
\eps'n$ (where the convergence depends on $c$).

Recall that we assumed $c\to C$. We show how to drop this assumption
here. Let $(c_i)_{i\in \setN}$ such that every $c_i\in[k+\eps, c+k'
  -\eps]$. Let $r(n)$ be the probability that neither
$W(\multiknm)\leq h(n)$ nor $W(\multiknm)=n$. Then every subsequence
of $(c_i)_{i\in \setN}$ has a subsequence that converges to some
constant $C_0$ and in that subsequence $r(n)\to 0$. So by the
subsubsequence principle $r(n)\to 0$. Since the probability that
$\multiknm$ is simple is $\Omega(1)$, we have that $W(\simpleknm)\leq
h(n)$ or $W(\simpleknm)=n$ \aas

\subsection{No small $k$-cores}

\label{sec:small}

\begin{lem}
  \label{lem:small}
  Let $C_0$ be a constant.  Suppose that $m = m(n)$ satisfies $kn \leq
  2m\leq C_0 n$. Then there exists a constant $\gamma$ such that
  \aas\ the graph obtained from $\multiknm$ by deleting an edge chosen
  uniformly at random  either has a $k$-core of size at least
  $\gamma n$ or its $k$-core is empty.
\end{lem}
\begin{proof}
  This is an application of a result by Luczak and Janson~\cite[Lemma
  5.1]{JansonLuczak07}: if a degree sequence $(\ds_n)_{n\in \setN}$
  satisfies $\sum_{i} e^{\alpha d_i} \leq Rn$ for constants $\alpha$
  and $R$, then there is a constant~$\gamma$ such that \aas\ no
  subgraph of $\multi(\ds)$ with less than $\gamma n$ vertices has
  average degree at least~$k$.  We set $\alpha < 1/3$ and we will
  choose $R$ later.  Let $\cDknm \subseteq \Dknm$ be the set of degree
  sequences $\ds$ such that $\sum_{i} e^{\alpha d_i} \leq Rn$. It
  suffices to show that the degree sequence $\ds = \ds(\multiknm)$ is
  in $\cDknm$ \aas

  Let $\Ys = (Y_1,\dotsc, Y_n)$ be such that the $Y_i$'s are
  independent random variables with distribution
  $\tpoisson{k}{\lambda_{k,c}}$. As already mentioned before, $\ds$
  has the same distribution of $\Ys$ conditioned upon the event that
  $\sum_i Y_i = 2m$. Using~\cite[Theorem 4]{PittelWormald03}, one can
  prove that $\prob{\sum_{i} Y_i = 2m} = \Omega(1/ \sqrt{n})$. 

  For $J_0$ big enough (depending only on~$C_0$), we have that
  $\lambda_{k,C_0}/J_0\leq e^{-1}$, which implies
  $\lambda_{k,c}/J_0\leq e^{-1}$.  Clearly, $\sum_{j\leq J_0} e^{\alpha
    j}D_j(\Ys) \leq e^{\alpha J_0} n$. Let $J_1 = J_0+ (1+\beta) \log n$
  with $\beta \in (1/2,(2\alpha)^{-1}-1)$. Let $p =
  \lambda_{k,c}^{J_0-1}/((J_0-1)!f_k(\lambda_{k,c}))$. Then
\begin{equation*}
  \begin{split}
    \prob{\exists j > J_1 \text{ with }D_j(\Ys)>0}
    &\leq
    n p\sum_{i\geq 0}\frac{1}{e^{(1+\beta)\log n+i}}  
    \leq
    \frac{p n^{-\beta}}{1-e^{-1}}
    = O(n^{-\beta}).
  \end{split}
\end{equation*}
Using the fact that $\prob{\sum_i Y_i = 2m} = \Omega(1 / \sqrt{n})$
and $\beta \in (1/2,(2\alpha)^{-1}-1)$, we conclude that $\prob{\max_i
  d_i > J_1} = o(1)$. And so $\sum_{j>J_1} e^{\alpha j}D_j(\Ys) = 0$
\aas

Now we consider $j\in (J_0, J_1]$. By Hoeffding's inequality and using
the fact that $\prob{\sum_i Y_i =2m} = \Omega(1/ \sqrt{n})$, we have
that $\prob{|D_j(\Ys) -p^{(j)}n|\geq a \sqrt{n}} = O(\sqrt n) e^{-a^2}
$, where $p^{(j)} = \lambda_{k,c}^j/(j!f_{k}(\lambda_{k,c}))$. Thus,
\begin{equation*}
  \begin{split}
    \prob[\Big]{|D_j(\Ys) -p^{(j)}n|\geq a \sqrt{n} \text{ for some }j\in(J_0,J_1]}
    &=
    (1+\beta) \log n O(\sqrt n) e^{-a^2}
    \\
    &=
  O(n^{-\beta'}),
  \end{split}
\end{equation*}
for $a = \sqrt{(1+\beta')\log n}$ with $\beta'>0$. Thus, \aas
\begin{equation*}
  \begin{split}
    \sum_{j=J_0+1}^{J_1}
    &e^{\alpha j} D_j(\Ys)
    \leq
    e^{\alpha J_0}
    \sum_{j=1}^{(1+\beta)\log n}
     e^{\alpha j} \left(p\frac{1}{e^j}n + a\sqrt{n}\right)
    \\
    &\leq
    e^{\alpha J_0}
    \left(
      n p \sum_{j=1}^{(1+\beta)\log n} e^{-2j/3}
    +
    e^{\alpha (J_1-J_0)}(J_1-J_0) a\sqrt{n}\right)
    \\
    &\leq
    e^{\alpha J_0}\left(
    \frac{p}{e^{2/3}(1-e^{-2/3})}
    + \frac{n^{(1+\beta)\alpha + 1/2}(\log n)^{3/2}}{n}\sqrt{1+\beta'}(1+\beta)
    \right) n.
    \end{split}
\end{equation*}
Using that $1+\beta < (2\alpha)^{-1}$, we can set $R = e^{\alpha
  J_0}(1 + p/(1-e^{-2/3})+\sqrt{1+\beta'}(1+\beta))$.
\end{proof}
 

\section{Working with simple graphs}
\label{sec:simple}

\begin{lem}
  \label{lem:simple_infty}
  Let $\eps > 0$ be a fixed real. Suppose that $c = 2m/n$ satisfies
  $k+\eps \leq c\leq c_k'-\eps$. Then there exists a function
  $h(n)\to\infty$ such that $\prob{W(\simpleknm) \geq h(n)} = \Omega(1)$.
\end{lem}

Together with Section~\ref{sec:de-inter}, this lemma implies that
$\prob{W(\simpleknm) = n} = \Omega(1)$, which completes the proof of
Theorem~\ref{thm:main}(ii).

We now prove Lemma~\ref{lem:simple_infty}. We will work with degree
sequences. Let $\simple(\ds)$ be chosen uniformly at random from all
(simple) $k$-cores with degree sequence~$\ds$. Let $\phi(n) = o(1)$
with $\phi(n) = \omega(n^{-1/4})$. Let $\Ys = (Y_1,\dotsc, Y_n)$ be
such that the $Y_i$'s are independent truncated Poisson variables with
parameters $(k, \lambda_{k, c})$.  Let $\tDknm$ be the degree
sequences $\ds$ such that $|D_k(\ds)-\esp{D_k(\Ys)}|\leq n\phi(n)$ and
$\max_i d_i \leq n^{\beta}$ for some $\beta \in(0,0.25)$ and
$|\eta(\ds) - \esp{\eta(\Ys)}| \leq \phi(n)$.  Similarly to the proof
in Section~\ref{sec:proofcor_super}, one can prove that
$\ds(\simpleknm)\in\tDknm$ \aas{} Thus, it suffices to show that,
there exists $h(n)\to\infty$ such that
\begin{equation*}
  \prob{W(\simple(\ds))\geq h(n)} = \Omega(1),
\end{equation*}
for $\ds\in\tDknm$.

For $\ds\in\tDknm$, we will couple deletion algorithms for $\simple(\ds)$
and $\multi(\ds)$ so that they coincide for $t(n) \to \infty$ steps.
We use a deletion algorithm that is essentially the same as the one we
used in the other sections. The only difference is that we explore a
whole vertex at a time (instead of an edge at a time) and mark the
vertices that have to be deleted. 

\noindent \textbf{Deletion procedure by vertex:}
\begin{itemize}
\item Iteration 0: Choose an edge $uv$ uniformly at random, delete
  $uv$ and mark the vertices with degree less than $k$.
\item Loop: While there is an undeleted marked vertex, say $w$, find
  its neighbours, delete $w$ and the edges incident to it, and then
  mark all neighbours of $w$ that now have degree less than $k$.
\end{itemize}
If we can do such a coupling for $t(n)\to\infty$ iterations of the loop,
then we can choose $h(n)\to\infty$ such $h(n)\leq
\min\set{t(n),\eps'n}$ with $\eps'$ as in
Theorem~\ref{thm:inter_degrees} so that $\prob{W(\multi(\ds))\geq
  h(n)} = \Omega(1)$. This would imply that the deletion algorithm did not
stop for at least $h(n)$ steps and so $\prob{W(\simple(\ds))\geq h(n)}
\geq \Omega(1)$.

In the rest of this section, we show that there exists $t(n)\to\infty$
such that we can couple the deletion algorithms for $\simple(\ds)$ and
$\multi(\ds)$ so that they coincide for $t(n)$ iterations of the
loop. For now assume that $t(n)\to\infty$ with $t(n)\leq \log
n$. Later we add more restrictions on the growth of $t(n)$.  We show
that the probabilities that a certain edge $uv$ is chosen in the first
step are asymptotically equivalent for $\simple(\ds)$ and
$\multi(\ds)$ and so the first step can be coupled. For the other
steps $i\leq t(n)$, we show that the probabilities that the set of
neighbours of the vertex $w$ is some specific set are again
asymptotically equivalent for $\simple(\ds)$ and $\multi(\ds)$ with
some error $\xi(n) = o(1)$. So we can couple the deletion algorithms
for $t(n)$ steps, where $t(n)$ will depend on $\xi(n)$. In the
computations in this section, we will use $\probmultiop$ to denote the
probabilities in the deletion procedure for $\multi(\ds)$ and we will
use $\probsimpleop$ to denote the probabilities in the deletion
procedure for $G(\ds)$. First we analyse the procedure for
multigraphs. Let $uv \in \binom{V}{2}$. Then
\begin{equation*}
  \begin{split}
    \probmulti{uv\text{ is chosen in the first step}}
    &=
    \frac{\probmulti{uv\in E(\multi(\ds))}}{m}
    \\
    &=
    \frac{d_ud_v}{m}\frac{(2m -2)!2^{ m}  m!}{(2 m)! 2^{ m-1}(m-1)!}
    \\
    &=
    \frac{d_ud_v}{m(2m-1)}
    =
    \frac{d_ud_v}{m(2m)}(1+\xi_1(n)),
  \end{split}
\end{equation*}
where $\xi_1(n) = o(1)$.

In $i$-th iteration of the loop, we delete a vertex $w$ and find its
set of neighbours $U$. Let $\ell$ be the current degree of $w$ and let
$\set{u_1,\dotsc, u_\ell}$ be a subset of $\ell$ undeleted
vertices. Let $x_1,\dotsc, x_\ell$ be an enumeration of the points
inside $w$. Let $y_1,\dotsc, y_\ell$ be the points matched to
$x_1\dotsc, x_\ell$. Let $\check m$ be the number of undeleted edges
at the beginning of the $i$-th iteration of the loop and let
$\check\ds$ be the degree sequence of the current graph. Using $[x]_j := (x)(x-1)\dotsc(x-j+1)$, we have
\begin{equation*}
  \prob{U = \set{u_1,\dotsc, u_\ell}}
  =
  \ell!
  \prob{y_i \in u_i \ \forall i}
  =
  \frac{\ell!\prod_{i=1}^\ell \check d_{u_i}}{2^\ell[\check m]_\ell}(1+\xi_2(n)),
\end{equation*}
where $\xi_2(n)= o(1)$ because $\check m\geq m-kt(n)\geq m-k\log n$.
Now we have to compute estimates for the probabilities in the
deletion algorithm for simple graphs. The following lemma is an
application of~\cite[Theorem 10]{McKay81}.

\begin{lem}
  \label{lem:mckay_main}
  Let $\ds\in\Dknm$ be such that $\max_i d_i\leq n^{0.25}$.
  Let $H$ be a graph on $[n]$ with at most $k t(n)$ edges. Let $L$
  be a supergraph of $H$ with at most $k$ edges more than $H$ such that
  there is a simple graph $G$ with degree sequence $\ds$ such that
  $G\cap L = H$. Then
  \begin{equation*}
    \probcond{L\subseteq \simple(\ds)}{H\subseteq \simple(\ds)}
    =
    \frac{\prod_{v=1}^{n}[d_i - h_i]_{j_i}}
    {2^{\card{E(J)}}[m]_{\card{E(J)}}} (1+ \nu(n)),
  \end{equation*}
  where $\hs$ is the degree sequence of $H$, $J = L-E(H)$, $\js$ is
  the degree sequence of $J$, and $\nu(n)=o(1)$.
\end{lem}
Notice that to use this lemma one has to check the existence of a
simple graph $G$ with certain properties. In our case, Erd\H os-Gallai
Theorem will be enough to ensure such simple graph exists.
\begin{lem}
  \label{lem:exists_simplegraph}
  Let $n$ be sufficiently large so that $n - n^{0.25} - k\log n >
  \sqrt{n}$.  Let $n'\geq n - \log n$. Let $\gs$ be a sequence on
  $[n']$ such that $g_1 \geq g_2\geq \dotsm\geq g_{n'}$, $\sum_i
  g_i$ is even, $g_1 \leq n^{0.25}$, $\card{\set{j\st g_j =
      0}}\leq k\log n$. Then there exists a simple graph with degree
  sequence $\gs$.
\end{lem}
The proofs for these lemmas are presented in
Section~\ref{sec:proofsimple}. Now we can analyse the deletion
algorithm for simple graphs.  Let $uv \in \binom{V}{2}$. Then
\begin{equation*}
  \probsimple{uv\text{ is chosen in the first step}}
  =
  \probsimple{uv\in E(\simple(\ds))}\frac{1}{m}.
\end{equation*}
We need to compute $\probsimple{uv\in E(\simple(\ds))}$. Note that this is
the same as $\probcond{L\subseteq G(\ds)}{H\subseteq G(\ds)}$ with $L
= ([n],\set{uv})$ and $H = ([n], \varnothing)$. In order to use
Lemma~\ref{lem:mckay_main}, we need to check if there is a simple
graph $G$ with $G\cap L = H$ with degree sequence $\ds$. This is the
same as saying that there exists a simple graph $G$ with degree
sequence $\ds$ such that $uv\not\in E(G)$. It suffices to show that,
for every set of vertices $S\subseteq [n]\setminus\set{u,v}$ of size
$d_v$, there is a simple graph with degree sequence $\ds'$, where
$\ds'$ is obtained from $\ds$ be deleting $v$ and decreasing the
degree of every vertex in $S$ by $1$ (that is, $S$ can be the set of
neighbours of $v$ and it does not include $u$).  Note that $\sum_i
d_i'$ is even because $\sum_j d_j$ is even. Moreover, $n-1 \geq
n-\log n$ and $\max_i d_i'\leq n^{0.25}$ and $\ds'$ has no
zeroes. By Lemma~\ref{lem:exists_simplegraph}, there is a simple graph
with degree sequence $\ds'$ and so we can use
Lemma~\ref{lem:mckay_main} to show that
\begin{equation*}
  \prob{uv\in \simple} = \frac{d_ud_v}{2 m} (1+\xi_3(n)),
\end{equation*}
where $\xi_3(n) = o(1)$. Now suppose we are in the $i$-th iteration of
the loop and deleting a vertex $w$. Let $\check n$ be the number of
undeleted vertices in the beginning of iteration $i$ and let $\check
m$ be the number of undeleted edges at the beginning of iteration
$i$. Let $\check\ds$ be the current degree sequence (that is, $\check
d_u$ is the number neighbours $u$ has among the undeleted vertices).
At each iteration we delete at most $k$ edges and only one vertex. So
$\check n \geq n - t(n)$ and $\check m \geq m - kt(n)$. Let
$\ell:=\check d_w$ and $\set{u_1,\dotsc, u_\ell}$ be a set with $\ell$
(undeleted) vertices.  Let $U$ be the neighbours of $w$ discovered in
iteration $i$. We want to compute the probability that
$U=\set{u_1,\dotsc, u_\ell}$.  In order to use
Lemma~\ref{lem:mckay_main}, we have to check if there exists a simple
graph $G$ with degree sequence $\check\ds$ such that $G\cap L = H$,
where $H$ is the graph discovered so far (which includes the deleted
vertices) and $L = ([n],E(H)\cup\set{wu_1,\dotsc,wu_\ell})$, which is
the same as checking if it is possible to get a simple graph such that
$w$ as no neighbours in $\set{u_1,\dots, u_\ell}$. Let $U'$ be a
set of $\ell$ undeleted vertices such that $U'\cap \set{u_1,\dotsc,
  u_\ell} = \varnothing$. There are plenty of choices for $U'$ since
$t(n)\leq \log(n)$. Let $\ds'$ be the degree sequence on $\check n -1$
obtained from $\check\ds$ by deleting $w$ and decreasing the degree of
each vertex in $U'$ by $1$. Then $\check n \geq n -\log n$, $\max_i
d_i'\leq n^{0.25}$ and $\card{\set{j\st d_j' = 0}}\leq kt(n)\leq k\log
n$. Using Lemma~\ref{lem:exists_simplegraph}, there is a simple graph
with degree sequence $\ds'$ and so, by Lemma~\ref{lem:mckay_main},
\begin{equation*}
  \prob{U = \set{u_1,\dotsc, u_\ell}}
  =
  \frac{\ell!\prod_{i=1}^\ell \check d_{u_i}}{2^\ell[\check m]_\ell}(1+\xi_4(n)),
\end{equation*}
where $\xi_4(n) = o(1)$. Thus, there exists a function $\xi(n)=o(1)$ such
that
\begin{equation*}
  \prob{U = \set{u_1,\dotsc, u_\ell}} = \probmulti{U =
    \set{u_1,\dotsc, u_\ell}}(1+\xi(n)) 
\end{equation*}
and, for every $uv\in \binom{V}{2}$,
\begin{equation*}
  \probsimple{uv\text{ is chosen in the first step}}
  \sim
  \probmulti{uv\text{ is chosen in the first step}}.
\end{equation*}
 We conclude that the deletion algorithms can be
 coupled for $t(n)$ steps as long as $(1+\xi)^t=1+o(1)$. Thus, it
 suffices to choose $t = o(1/\xi)$.

\subsection{Proofs of Lemma~\ref{lem:mckay_main} and
  Lemma~\ref{lem:exists_simplegraph}}
\label{sec:proofsimple}
\begin{proof}[Proof of Lemma~\ref{lem:mckay_main}]
  Let $\Delta_L$ be the maximum degree in $L$ and let $\Delta$ be the
  maximum degree in $D$. Note that $\Delta_L \leq |E(H)|+k \leq k\log
  n + k$ and $\Delta \leq n^{0.25}$. Then
  \begin{equation*}
    \begin{split}
      \card{E(\simple(\ds))}-\card{E(H)}-\card{E(J)} 
      &\geq m - kt(n)-k \geq
      n^{0.25}(2n^{0.25})
      \\
      &\geq \Delta (\Delta+\Delta_L) =: D.
    \end{split}
  \end{equation*}
  So we can use part (a) of~\cite[Theorem 2.10]{McKay81} to obtain that
  \begin{equation*}
    \begin{split}
      \probcond{L\subseteq \simple(\ds)}{H\subseteq \simple(\ds)}
      &
      \leq
      \frac{\prod_{v=1}^{ n}[d_i - h_i]_{j_i}}
      {2^{\card{E(J)}}[ m - \card{E(H)}-D]_{\card{E(J)}}}
      \\&=
      \frac{\prod_{v=1}^{ n}[d_i - h_i]_{j_i}}
      {2^{\card{E(J)}}[ m]_{\card{E(J)}}}(1+\nu_1(n))
    \end{split}
  \end{equation*}
  with $\nu_1(n) = o(1)$ because $|E(J)|\leq k$ and $ m -
  \card{E(H)}-D \geq m - k\log n- 2\sqrt{n}$. Now we will use part (b)
  of Theorem 2.10 from~\cite{McKay81}. We have to check the conditions
  for (b):
  \begin{equation*}
    \begin{split}
      \card{E(\simple(\ds))}-\card{E(H)}-\card{E(J)} 
      &\geq 
       m -
      kt(n)-k \geq 
      k\geq n^{0.25}(n^{0.25}+1)
      \\
      &\geq 
       \Delta(\Delta+\Delta_L+2)
      +\Delta (\Delta_L+1),
    \end{split}
  \end{equation*}
  so we can apply~\cite{McKay81}[Theorem 2.10(b)]. We have to bound some
  errors given by~\cite{McKay81}[Theorem 2.10(b)]. We have that
  \begin{equation*}
    \begin{split}
      0&\leq \frac{\Delta (\Delta_L+1)}
    { m-\card{E(H)}-\card{E(J)} -  
      \Delta (\Delta+\Delta_L+2)}
    \\
    &\leq
    \frac{n^{0.25}(n^{0.25}+1)}
    {n - k\log n - n^{0.25}(2n^{0.25}+2)}
    =:
    \nu_2(n),
    \end{split}
  \end{equation*}
  with $\nu_2(n) = o(1)$, and
  \begin{equation*}
    \begin{split}
      0&\leq
      \frac{\Delta^2}
      {2(\card{E(G)}-\card{E(H)}-D-(1-1/e)\card{E(J)})}
      \\
      &\leq
      \frac{\sqrt{n}}
      {2(n -k\log n-  n^{0.25}(2n^{0.25}) -(1-1/e)k)}
      =: \nu_3(n),
    \end{split}
  \end{equation*}
  with $\nu_3(n) = o(1)$  Then~\cite{McKay81}[Theorem 2.10(b)] implies that
  \begin{equation*}
    \begin{split}
      \probcond{L\subseteq \simple(\ds)}{H\subseteq \simple(\ds)}
      &\geq
      \frac{\prod_{v=1}^{ n}[d_i - h_i]_{j_i}}
      {2^{\card{E(J)}}[ m]_{\card{E(J)}}}
      (1+\nu_4(n))\left(\frac{1+\nu_2(n)}{1+\nu_3(n)}\right)^{E(J)}.
    \end{split}
  \end{equation*}
  with $\nu_4(n) = o(1)$. Since $\nu_i(n)=o(1)$ for $i=1,2,3,4$, we
  can conclude that
  \begin{equation*}
    \probcond{L\subseteq \simple(\ds)}{H\subseteq \simple(\ds)}
      =
      \frac{\prod_{v=1}^{ n}[d_i - h_i]_{j_i}}
      {2^{\card{E(J)}}[ m]_{\card{E(J)}}}
      (1+\nu(n)),
  \end{equation*}
where $\nu = o(1)$.
\end{proof}

\begin{proof}[Proof of Lemma~\ref{lem:exists_simplegraph}]
  We will use Erd\H os-Gallai Theorem: $\gs$ is the degree sequence of
  a simple graph iff, for every $1\leq \ell \leq n'$,
  \begin{equation*}
    \sum_{i=1}^{\ell} g_i
    \leq
    \ell(\ell-1)
    +
    \sum_{j=\ell+1}^{n'}\min\set{\ell, g_j}.
  \end{equation*}
  If $\ell \geq n^{0.25}+1$, then $\sum_{i=1}^{\ell} g_i \leq \ell
  g_1 \leq \ell(\ell-1)$.  If $\ell \leq n^{0.25}$,
  \begin{equation*}
    \begin{split}
      \sum_{i=1}^{\ell} g_i
      &\leq
      \ell n^{0.25}
      \leq 
      \sqrt{n}
      \leq
      n - n^{0.25} - k\log n
      \leq
      n' - \ell - \card{\set{j\st g_j = 0}} 
      \\
      &=
      \sum_{j= \ell+1}^{n'} 1 - \card{\set{j\st g_j = 0}} 
      \leq
      \sum_{j= \ell+1}^{n'}\min\set{\ell, g_j}.
    \end{split}
  \end{equation*}
\end{proof}

\section*{Acknowledgments}
I would like to thank my advisor Nicholas Wormald for his supervision
during this project.

\bibliographystyle{plain}
\bibliography{kcore}
 		
\end{document}